\pdfoutput=1
\RequirePackage{ifpdf}
\ifpdf 
\documentclass[pdftex]{sigma}
\else
\documentclass{sigma}
\fi

\usepackage{braket}
\usepackage[italicdiff]{physics}

\newcommand{\opcl}[2]{(#1, #2]}

\numberwithin{equation}{section}

\newtheorem{Theorem}{Theorem}[section]
\newtheorem{Corollary}[Theorem]{Corollary}
\newtheorem{Proposition}[Theorem]{Proposition}
 { \theoremstyle{definition}
\newtheorem{Remark}[Theorem]{Remark} }

\begin{document}

\allowdisplaybreaks

\newcommand{\arXivNumber}{1808.02209}

\renewcommand{\PaperNumber}{005}

\FirstPageHeading

\ShortArticleName{A Constraint on Chern Classes of Strictly Pseudoconvex CR Manifolds}

\ArticleName{A Constraint on Chern Classes\\ of Strictly Pseudoconvex CR Manifolds}

\Author{Yuya TAKEUCHI}

\AuthorNameForHeading{Y.~Takeuchi}

\Address{Department of Mathematics, Graduate School of Science, Osaka University,\\ 1-1 Machikaneyama-cho, Toyonaka, Osaka 560-0043, Japan}
\Email{\href{mailto:yu-takeuchi@cr.math.sci.osaka-u.ac.jp}{yu-takeuchi@cr.math.sci.osaka-u.ac.jp}, \href{mailto:yuya.takeuchi.math@gmail.com}{yuya.takeuchi.math@gmail.com}}
\URLaddress{\url{https://sites.google.com/view/yuya-takeuchi-english/home}}

\ArticleDates{Received July 25, 2019, in final form January 18, 2020; Published online January 21, 2020}

\Abstract{This short paper gives a constraint on Chern classes of closed strictly pseudoconvex CR manifolds (or equivalently, closed holomorphically fillable contact manifolds) of dimension at least five. We also see that our result is ``optimal'' through some examples.}


\Classification{32V15; 32T15; 32V05; 57R17}

\section{Introduction} \label{section:introduction}

Let $\big(M, T^{1, 0} M\big)$ be a closed connected strictly pseudoconvex CR manifold of dimension $2 n + 1\allowbreak \geq 5$. It is known that $M$ can be realized as the boundary of a strictly pseudoconvex domain in a complex manifold~\cite{Boutet-de-Monvel75,Harvey-Lawson75,Lempert95}. This fact gives some restrictions to its topology. For example, Bungart~\cite{Bungart92} has proved that the cup product
\begin{gather*}
H^{i_{1}}(M, \mathbb{C}) \otimes \dots \otimes H^{i_{m}}(M, \mathbb{C}) \to H^{\abs{I}}(M, \mathbb{C})
\end{gather*}
vanishes for any multi-index $I = (i_{1}, \dots , i_{m}) \in \mathbb{N}_{+}^{m}$ satisfying $i_{l} \leq n - 1$ and $\abs{I} = i_{1} + \dots + i_{m} \geq n + 2$; see also \cite{Popescu-Pampu08}. As noted in the last paragraph of \cite{Bungart92}, this result follows also from a result of $L^{2}$ Hodge theory on strictly pseudoconvex domains by Ohsawa~\cite{Ohsawa81,Ohsawa82}.

In this paper, we will point out that Ohsawa's result gives also a constraint on Chern classes of strictly pseudoconvex CR manifolds. For a complex vector bundle $E$ and a multi-index $K = (k_{1}, \dots , k_{m}) \in \mathbb{N}_{+}^{m}$, we denote by $c_{K}(E)$ the cohomology class $c_{k_{1}}(E) \dotsm c_{k_{m}}(E)$.

\begin{Theorem} \label{thm:vanishing-of-Chern-classes}Let $\big(M,T^{1,0} M\big)$ be a closed strictly pseudoconvex CR manifold of dimension $2 n + 1 \geq 5$. Then $c_{K}\big(T^{1,0}M\big)$ vanishes in $H^{2 \abs{K}}(M, \mathbb{C})$ for any multi-index $K$ with $2 \abs{K} \geq n + 2$.
\end{Theorem}

Note that sharper results than Theorem~\ref{thm:vanishing-of-Chern-classes} hold for some particular cases. For instance, Theorem~\ref{thm:vanishing-of-Chern-classes} holds for $\mathbb{Z}$-coefficients if $M$ can be realized as a real hypersurface in a Stein manifold. On the other hand, we can relax the degree condition to $2 \abs{K} \geq n + 1$ if $M$ is the link of an isolated singularity; see~\cite{Kollar13} for example. However, we will see in Section~\ref{section:examples} that Theorem~\ref{thm:vanishing-of-Chern-classes} is ``optimal'' for a general closed strictly pseudoconvex CR manifold. More precisely, we will give an example of a closed strictly pseudoconvex CR manifold $\big(M,T^{1,0} M\big)$ of dimension $2 n + 1$ with $c_{K}\big(T^{1,0}M\big) \neq 0$ in $H^{2 \abs{K}}(M, \mathbb{Z})$ for any multi-index $K$ with $1 \leq \abs{K} \leq n$ (Proposition~\ref{prop:non-vanishing-torsion}). We will also construct a $(2 n + 1)$-dimensional closed strictly pseudoconvex CR manifold $\big(M,T^{1,0} M\big)$ with $c_{K}\big(T^{1,0}M\big) \neq 0$ in $H^{2 \abs{K}}(M, \mathbb{C})$ for any multi-index $K$ with $1 \leq 2\abs{K} \leq n + 1$ (Proposition~\ref{prop:optimal-degree}).

Theorem~\ref{thm:vanishing-of-Chern-classes} in particular implies that $c_{2}\big(T^{1,0}M\big) = 0$ in $H^{4}(M, \mathbb{C})$ for any closed strictly pseudoconvex CR manifold $\big(M,T^{1,0} M\big)$ of dimension five. Hence the assumption on the second Chern class in \cite[Proposition~8.8]{Case-Gover17} automatically holds for the strictly pseudoconvex case, which is compatible with \cite[Theorem~4.6]{Marugame16}. It would be an interesting problem whether Theo\-rem~\ref{thm:vanishing-of-Chern-classes} remains true for non-degenerate CR manifolds or non-compact strictly pseudoconvex CR manifolds.

We also remark a relation between Theorem~\ref{thm:vanishing-of-Chern-classes} and contact topology. Let $M$ be a closed oriented $(2 n + 1)$-dimensional manifold with a positive co-oriented contact structure $\xi$. Then we can define the $k$-th Chern class $c_{k}(\xi) \in H^{2 k}(M, \mathbb{Z})$ of $\xi$ by using an adapted almost complex structure on~$\xi$. A contact structure $\xi$ is said to be \emph{holomorphically fillable} if $M$ can be realized as the boundary of a strictly pseudoconvex domain and $\xi = \Re T^{1, 0} M$ holds. In this case, $c_{k}(\xi)$ coincides with $c_{k}\big(T^{1,0}M\big)$. Therefore Theorem~\ref{thm:vanishing-of-Chern-classes} implies the following

\begin{Corollary}Let $(M, \xi)$ be a holomorphically fillable contact manifold of dimension $2 n + 1 \geq 5$. Then $c_{K}(\xi)$ is a torsion class in $H^{2 \abs{K}}(M, \mathbb{Z})$ for any multi-index $K$ satisfying $2 \abs{K} \geq n + 2$.
\end{Corollary}

This result is no longer true for a general contact structure. As far as the author knows, such a constraint on Chern classes has not been obtained yet for a contact structure
satisfying other fillability conditions.

This paper is organized as follows. In Section~\ref{section:preliminaries}, we recall some basic notions related to strictly pseudoconvex domains and CR manifolds.
Section~\ref{section:proof-of-theorem} provides a proof of Theorem~\ref{thm:vanishing-of-Chern-classes}. In Section~\ref{section:examples}, we will see that, by constructing some examples, Theorem~\ref{thm:vanishing-of-Chern-classes} does not hold for $\mathbb{Z}$-coefficients in general, and the degree condition is optimal.

\section{Preliminaries} \label{section:preliminaries}

Let $X$ be a complex manifold of complex dimension $n + 1$. A real-valued smooth function $\varphi$ on $X$ is called an \emph{exhaustion function} if $\varphi^{- 1}(\opcl{- \infty}{c})$ is compact in $X$ for any $c \in \mathbb{R}$. We say $\varphi$ to be \emph{strictly plurisubharmonic} if its complex Hessian $\pqty{\pdv*{\varphi}{z^{i}}{\overline{z}^{j}}}$ is a positive-definite Hermitian matrix for any holomorphic local coordinate $\big(z^{1}, \dots , z^{n + 1}\big)$. A \emph{Stein manifold} is a~complex manifold admitting a~strictly plurisubharmonic exhaustion function.

Let $\Omega$ be a relatively compact domain in $X$ with smooth boundary $M = \partial \Omega$. Then there exists a smooth function $\rho$ on $X$ such that
\begin{gather*}
\Omega = \rho^{-1}((- \infty, 0)), \qquad M = \rho^{- 1}(0), \qquad
{\rm d} \rho \neq 0 \ \text{on} \ M;
\end{gather*}
such a $\rho$ is called a \emph{defining function} of $\Omega$. A domain $\Omega$ is said to be \emph{strictly pseudoconvex} if we can take a defining function $\rho$ of $\Omega$ that is strictly plurisubharmonic near $M$. We call $\Omega$ a \emph{Stein domain} if $\Omega$ admits a defining function $\rho$ that is strictly plurisubharmonic on a~neighborhood of the closure of $\Omega$. Note that a Stein domain is a Stein manifold; this is because $- 1 / \rho$ defines a strictly plurisubharmonic exhaustion function on~$\Omega$.

We next recall some notions related to CR manifolds. Let $M$ be a $(2 n + 1)$-dimensional manifold without boundary. A \emph{CR structure} is an $n$-dimensional complex subbundle $T^{1, 0} M$ of the complexified tangent bundle $T M \otimes \mathbb{C}$ satisfying the following two conditions:
\begin{gather*}
T^{1, 0} M \cap \overline{T^{1, 0} M} = 0, \qquad
\big[\Gamma\big(T^{1,0}M\big), \Gamma\big(T^{1,0}M\big)\big] \subset \Gamma\big(T^{1,0}M\big).
\end{gather*}
A typical example of a CR manifold is a real hypersurface $M$ in a complex manifold $X$; it has the natural CR structure
\begin{gather*}
T^{1, 0} M= T^{1, 0} X |_{M} \cap (T M \otimes \mathbb{C}).
\end{gather*}
A CR structure $T^{1, 0} M$ is said to be \emph{strictly pseudoconvex} if there exists a nowhere-vanishing real one-form $\theta$ on $M$ such that $\theta$ annihilates $T^{1, 0} M$ and
\begin{gather*}
- \sqrt{- 1} {\rm d} \theta \big(Z, \overline{Z}\big) > 0, \qquad 0 \neq Z \in T^{1, 0} M;
\end{gather*}
we call such a one-form a \emph{contact form}. The natural CR structure on the boundary of a strictly pseudoconvex domain is strictly pseudoconvex. Conversely, any closed connected strictly pseudoconvex CR manifold of dimension at least five can be realized as the boundary of a strictly pseudoconvex domain.

\section{Proof of Theorem~\ref{thm:vanishing-of-Chern-classes}} \label{section:proof-of-theorem}

Let $\big(M,T^{1,0} M\big)$ be a closed strictly pseudoconvex CR manifold of dimension $2 n + 1 \geq 5$. Without loss of generality, we may assume that $M$ is connected.
Then it can be realized as the boundary of a strictly pseudoconvex domain $\Omega$ in an $(n + 1)$-dimensional complex projective manifold $X$~\cite[Theorem~8.1]{Lempert95}.
The vector bundle $T^{1, 0} X |_{M}$ is decomposed into the direct sum of $T^{1, 0} M$ and a~trivial complex line bundle, and consequently $c_{K}\big(T^{1,0}M\big) = c_{K}\big(T^{1, 0} X |_{M}\big)$ is in the image of the restriction morphism $H^{2 \abs{K}}(\Omega, \mathbb{C}) \to H^{2 \abs{K}}(M, \mathbb{C})$. On the other hand, the natural morphism $H^{i}_{c}(\Omega, \mathbb{C}) \to H^{i}(\Omega, \mathbb{C})$ is surjective for $i \geq n + 2$ according to \cite[Corollary~4]{Ohsawa82}; see \cite[Lemma]{Bungart92} for another proof. From the exact sequence
\begin{gather*}
H^{i}_{c}(\Omega, \mathbb{C}) \to H^{i}(\Omega, \mathbb{C}) \to H^{i}(M, \mathbb{C}),
\end{gather*}
it follows that $H^{i}(\Omega, \mathbb{C}) \to H^{i}(M, \mathbb{C})$ is identically zero for $i \geq n + 2$. This completes the proof of Theorem~\ref{thm:vanishing-of-Chern-classes}.

\section{Examples} \label{section:examples}

In this section, we treat some examples related to Theorem~\ref{thm:vanishing-of-Chern-classes}. We first show that Theorem~\ref{thm:vanishing-of-Chern-classes} does not hold for $\mathbb{Z}$-coefficients in general.

\begin{Proposition} \label{prop:non-vanishing-torsion} For each positive integer $n$, there exists a closed strictly pseudoconvex CR manifold $\big(M,T^{1,0} M\big)$ of dimension $2 n + 1$ such that, for every multi-index $K$ with $1 \leq \abs{K} \leq n$, the $K$-th Chern class $c_{K}\big(T^{1,0}M\big)$ is not equal to zero in $H^{2 \abs{K}}(M, \mathbb{Z})$.
\end{Proposition}

\begin{proof}Let $\mathbb{C}P^{n}$ be the $n$-dimensional complex projective space and $\mathcal{O}(- d)$ be the holomorphic line bundle over $\mathbb{C}P^{n}$ of degree $- d < 0$. Fix a Hermitian metric $h$ on $\mathcal{O}(- d)$ with negative curvature, and consider
\begin{gather*}
M = \big\{v \in \mathcal{O}(- d) \,|\, h(v, v) = 1\big\},
\end{gather*}
which is a principal $S^{1}$-bundle over $\mathbb{C}P^{n}$. This manifold has a strictly pseudoconvex CR structure $T^{1, 0} M$ induced from the total space of $\mathcal{O}(- d)$, and $T^{1, 0} M$ is isomorphic to $\pi^{*} T^{1, 0} \mathbb{C}P^{n}$ as a~complex vector bundle, where $\pi \colon M \to \mathbb{C}P^{n}$ is the canonical projection. Consider the Gysin exact sequence
\begin{gather*}
H^{i - 2}\big(\mathbb{C}P^{n}, \mathbb{Z}\big) \xrightarrow{- d \cdot \alpha}H^{i}\big(\mathbb{C}P^{n}, \mathbb{Z}\big) \xrightarrow{\pi^{*}} H^{i}(M, \mathbb{Z}) \rightarrow H^{i - 1}\big(\mathbb{C}P^{n}, \mathbb{Z}\big),
\end{gather*}
where $\alpha = c_{1}(\mathcal{O}(1))$ is a generator of $H^{2}\big(\mathbb{C}P^{n}, \mathbb{Z}\big) \cong \mathbb{Z}$. This gives that
\begin{gather*}
H^{2 k}(M, \mathbb{Z})\cong\begin{cases}
\mathbb{Z}, & k = 0, \\
\mathbb{Z} / d \mathbb{Z}, & 1 \leq k \leq n, \\
0, & \text{otherwise},
\end{cases}
\end{gather*}
and $\pi^{*} \alpha^{k}$ is a generator of $H^{2 k}(M, \mathbb{Z})$. On the other hand,
\begin{gather*}
c_{K}\big(T^{1,0}M\big)= \pi^{*} c_{K}\big(T^{1, 0} \mathbb{C}P^{n}\big) = \bqty{\prod_{l = 1}^{m} \binom{n + 1}{k_{l}}} \pi^{*} \alpha^{\abs{K}}.
\end{gather*}
Hence, for a multi-index $K$ with $1 \leq \abs{K} \leq n$, the $K$-th Chern class $c_{K}\big(T^{1,0}M\big)$ is equal to zero in $H^{2 \abs{K}}(M, \mathbb{Z})$ if and only if $\prod\limits_{l = 1}^{m} \binom{n + 1}{k_{l}} \in d \mathbb{Z}$. In particular if we choose $d$ as a prime integer greater than $n + 1$, then $c_{K}\big(T^{1,0}M\big)$ does not vanish for any $K$ with $1 \leq \abs{K} \leq n$.
\end{proof}

We next see that the degree condition in Theorem~\ref{thm:vanishing-of-Chern-classes} is optimal.

\begin{Proposition} \label{prop:optimal-degree}For each positive integer $n$, there exists an $(n + 1)$-dimensional Stein domain~$\Omega$ such that its boundary $M = \partial \Omega$ satisfies $c_{K}\big(T^{1,0}M\big) \neq 0$ in $H^{2 \abs{K}}(M, \mathbb{C})$ for any multi-index~$K$ with $1 \leq 2 \abs{K} \leq n + 1$.
\end{Proposition}

\begin{proof}Let $\Omega_{0}$ be a Stein domain in a two-dimensional complex manifold $X_{0}$ such that its boundary $M_{0} = \partial \Omega_{0}$ satisfies $c_{1}\big(T^{1, 0} M_{0}\big) \neq 0$ in $H^{2}(M_{0}, \mathbb{C})$; see~\cite[Theorem~6.2]{Etnyre-Ozbagci08} for an example of such $\Omega_{0}$. Take a defining function $\rho$ of $\Omega_{0}$ that is strictly plurisubharmonic near the closure of~$\Omega_{0}$. Without loss of generality, we may assume that $\rho$ is an exhaustion function on $X_{0}$. Then, for sufficiently small $\epsilon$, there exists a diffeomorphism $\chi \colon (- \epsilon, \epsilon) \times M_{0} \to \rho^{- 1}((- \epsilon, \epsilon))$ satisfying $\chi(0, p) = p$ and $\rho \circ \chi (t, p) = t$. The function $\psi_{0} = - 1 / \rho$ gives a strictly plurisubharmonic exhaustion function on~$\Omega_{0}$.

We first show the statement for the case of $n = 2 k - 1$. Consider the domain
\begin{gather*}
\Omega= \big\{(p_{1}, \dots , p_{k}) \in (\Omega_{0})^{k} \,|\, \psi_{0}(p_{1}) + \dots + \psi_{0}(p_{k}) < 2 k / \epsilon \big\}.
\end{gather*}
The function $\psi(p_{1}, \dots , p_{k}) = \psi_{0}(p_{1}) + \dots + \psi_{0}(p_{k})$ is a strictly plurisubharmonic exhaustion function on $(\Omega_{0})^{k}$, and $d \psi \neq 0$ on $M = \partial \Omega$. Hence $\Omega$ is a Stein domain in $(\Omega_{0})^{k} \subset (X_{0})^{k}$. As noted in the proof of Theorem~\ref{thm:vanishing-of-Chern-classes}, the cohomology class $c_{K}\big(T^{1,0}M\big)$ coincides with $c_{K}\big(T^{1, 0} (X_{0})^{k} |_{M}\big)$. Consider the map $\iota \colon (M_{0})^{k} \to M \big({\subset} (X_{0})^{k}\big)$ defined by
\begin{gather*}
\iota(p_{1}, \dots , p_{k})= (\chi(- \epsilon / 2, p_{1}), \dots , \chi(- \epsilon / 2, p_{k})).
\end{gather*}
Since this map is homotopic to the natural embedding $(M_{0})^{k} \hookrightarrow (X_{0})^{k}$,
\begin{gather*}
\iota^{*} c_{K}\big(T^{1,0}M\big)= c_{K}\big(\iota^{*} T^{1, 0} (X_{0})^{k}\big)= c_{K}\big(T^{1, 0} (X_{0})^{k} |_{(M_{0})^{k}}\big)= c_{K}\big(\big(T^{1, 0} M_{0}\big)^{k}\big).
\end{gather*}
From the assumption on $\Omega_{0}$, it follows that $c_{K}\big(\big(T^{1, 0} M_{0}\big)^{k}\big) \neq 0$ in $H^{2 \abs{K}}\big((M_{0})^{k}, \mathbb{C}\big)$ for any multi-index $K$ with $1 \leq \abs{K} \leq k$, and consequently $c_{K}\big(T^{1,0}M\big) \neq 0$ in $H^{2 \abs{K}}(M, \mathbb{C})$.

We next treat the case of $n = 2 k$. Consider the domain
\begin{gather*}
\Omega= \big\{(p_{1}, \dots , p_{k}, z) \in (\Omega_{0})^{k} \times \mathbb{C}\,\big|\, \psi_{0}(p_{1}) + \dots + \psi_{0}(p_{k}) + |z|^{2} < 2 k / \epsilon\big\}.
\end{gather*}
This $\Omega$ is a Stein domain in $(\Omega_{0})^{k} \times \mathbb{C} \subset (X_{0})^{k} \times \mathbb{C}$.
Consider the map $\iota \colon (M_{0})^{k} \to M = \partial \Omega$ given by
\begin{gather*}
\iota(p_{1}, \dots , p_{k})= (\chi(- \epsilon / 2, p_{1}), \dots , \chi(- \epsilon / 2, p_{k}), 0).
\end{gather*}
Then we obtain
\begin{gather*}
\iota^{*} c_{K}\big(T^{1,0}M\big)= c_{K}\big(\big(T^{1, 0} M_{0}\big)^{k}\big) \neq 0
\end{gather*}
in $H^{2 \abs{K}}\big((M_{0})^{k}, \mathbb{C}\big)$ for any $K$ satisfying $1 \leq \abs{K} \leq k$ in a similar way to the case of odd~$n$. This proves the statement.
\end{proof}

\begin{Remark}Cao and Chang have stated that, for any Stein domain $\Omega$, the first Chern class $c_{1}\big(T^{1,0}\Omega\big)$ vanishes in $H^{2}(\Omega, \mathbb{C})$; see Proposition~3.2(1) and the proof of Proposition~2.4 in~\cite{Cao-Chang07}. This statement contradicts the above example since $c_{1}\big(T^{1,0}M\big)$ is the image of~$c_{1}\big(T^{1,0}\Omega\big)$ by the restriction morphism. Unfortunately, their proof contains an error. They have claimed that the vanishing of $H^{1, 1}(\Omega)$ would imply that of the first Chern class. However, this is not true since $H^{2}(\Omega, \mathbb{C})$ has no Hodge decomposition like on closed K\"{a}hler manifolds. In particular, we need to add some assumptions on the first Chern class to Main Theorems~(1) and~(2) in~\cite{Cao-Chang07} at least.
\end{Remark}

\subsection*{Acknowledgements}
This paper is a part of the author's Ph.D.~Thesis. He is grateful to his supervisor Kengo Hirachi for various valuable suggestions. He also would like to thank Masanori Adachi, Tomohiro Asano, and Takeo Ohsawa for helpful comments. This work was supported by JSPS Research Fellowship for Young Scientists, JSPS KAKENHI Grant Numbers JP16J04653 and JP19J00063, and the Program for Leading Graduate Schools, MEXT, Japan.

\pdfbookmark[1]{References}{ref}
\LastPageEnding

\end{document}